\documentclass[a4paper]{amsart}

\usepackage{amsthm,amsfonts,amsmath, amssymb}

\let\oldmarginpar\marginpar
\renewcommand{\marginpar}[1]{\oldmarginpar{\small\textit{{#1}}}}

\usepackage{setspace}
\usepackage{graphicx,float,enumerate,subfigure,tikz}
\usepackage[ruled,boxed,commentsnumbered,norelsize]{algorithm2e}
\usepackage{verbatim}
\usepackage{url}   
\usepackage[capitalise]{cleveref}
\usepackage[sort&compress,comma]{natbib}
 \usepackage{enumitem}
\setlist[enumerate,1]{label=(\roman*),font=\normalfont}

\newtheorem{theorem}{Theorem}
\newtheorem{corollary}[theorem]{Corollary}

\newtheorem{proposition}[theorem]{Proposition}

\theoremstyle{definition}

\theoremstyle{remark}

\crefname{remark}{Remark}{Remarks}
\crefname{claim}{Claim}{Claims}

\theoremstyle{definition}

\usepackage{etoolbox}
\AtEndEnvironment{proof}{\setcounter{claim}{0}}

\theoremstyle{remark}

\usepackage{etoolbox}
\AtEndEnvironment{proof}{\setcounter{case}{0}}
\crefname{rmk}{Remark}{Remarks} 
\crefname{problem}{Problem}{Problems}
\crefname{conjecture}{Conjecture}{Conjectures}

\usepackage{mathtools}

\DeclareMathOperator{\sdiff}{\mathbin{\triangle}} %symmetric difference

\newcommand{\lset}[1]{\ensuremath{\{#1\}}} %inline set
\newcommand{\R}{\mathbb{R}}

\newcommand{\C}{\mathcal{C}}
\newcommand{\B}{\mathcal{B}}

\renewcommand{\ge}{\geqslant} 
\renewcommand{\le}{\leqslant}

\usepackage[f]{esvect}%Eight differents arrows are available: a b c d e f g h 

\renewcommand{\v}{\null}

\date{\today}
\title[Cycle space of polytopal graphs]{Cycle space of graphs of polytopes} 
 
\usepackage{amsaddr}

\author{Guillermo Pineda-Villavicencio}
\address{Federation University,  Australia\\School of Information Technology, Deakin University, Geelong,  Australia} 
\email{\texttt{work@guillermo.com.au}}

 
\begin{document}
\begin{abstract} 
It is folklore that the cycle space of graphs of polytopes is generated by the cycles bounding the 2-faces. We provide a proof of this result that bypass homological arguments, which seem to be the most widely known proof.  As a corollary, we obtain a result of \citet{Blind1994} stating that  graphs of polytopes  are bipartite if and only if  graphs of every 2-face  are bipartite.     \end{abstract}

\maketitle  
 
\section{Introduction} 
A (convex) polytope is the convex hull of a finite set $X$ of points in $\R^{d}$.  The \textit{dimension} of a polytope in $\R^{d}$ is one less than the maximum number of affinely independent points in the polytope, and a polytope of dimension $d$ is referred to as a \textit{$d$-polytope}. A {\it face} of a polytope $P$ in $\R^{d}$ is $P$ itself, or  the intersection of $P$ with a hyperplane in $\R^{d}$ that contains $P$ in one of its closed halfspaces.  A face of dimension 0, 1, and $d-1$ in a $d$-polytope is a \textit{vertex}, an {\it edge}, and a {\it facet}, respectively. The set of vertices and edges of a polytope or a graph are denoted by $V$ and $E$, respectively. The \textit{graph} $G(P)$ of a polytope  $P$ is the abstract graph with vertex set $V(P)$ and edge set $E(P)$.   

Let $G$ be a graph, and let $G'$ and $G''$ be two spanning subgraphs of $G$. The  \textit{symmetric difference} $\sdiff$ of $G'$ and $G''$ is the spanning subgraph of $G$ whose edge set is the symmetric difference of $E(G')$ and $E(G'')$. A spanning subgraph $G'$ of a graph $G$ is an \textit{even subgraph} if every vertex of $G'$ has even degree. Every cycle in $G$ can be regarded as an even subgraph if  enough isolated vertices of $G$ are added. The set of all even subgraphs of a graph $G$ forms a vector space  $Z(G)$ over the field $GF(2)$, the 2-element field, with respect to the symmetric difference of spanning subgraphs; the space $Z(G)$ is called the \textit{cycle space} of $G$. And  the cycles of $G$, viewed as even subgraphs, span $Z(G)$. Thus, each even subgraph is the symmetric difference of cycles. See, for instance, in \citet[Sec.~1.9]{Die05}.

%\begin{proposition}[Even subgraphs] The following hold. %
%\begin{enumerate}
%\item The symmetric difference of two even subgraphs of a graph  is an even subgraph of the graph. 
%\item A subgraph of a graph is an even subgraph if and only if it is an edge-disjoint union of cycles viewed as even subgraphs.
%\end{enumerate}    
%\label{prop:Graphs-even-subgraphs-basic} 
%\end{proposition}

A cycle in a plane graph or a graph of a polytope is \textit{facial} if it bounds a face of the plane graph or a 2-face of the polytope.   
It is well known that every cycle in a  2-connected plane graph  is the symmetric difference of facial cycles  and that a 2-connected plane graph is bipartite if and only if every facial cycle is bipartite; see, for instance, \citet[Thm.~2.2.3, Cor.~2.4.6]{MohTho01}.  We provide elementary proofs of extensions of  these results to graphs of polytopes of all dimensions.   These extensions read as follows. 
%; a nonempty graph is \textit{connected} or \textit{1-connected}  if there is a path between every two of its vertices, and a graph with at least $d+1$ vertices is \textit{$d$-connected} if removing any $d-1$ vertices leaves a connected subgraph. 

%A collection $B$ of cycles in $G$ (viewed as even subgraphs) is a \textit{2-basis}\index{2-basis}\index{graph!2-basis} of $G$  if it is a basis of $Z(G)$ and every edge of $G$ is contained in at most two elements of $B$. \citet{Mac37} characterised planar graphs in terms of 2-bases.

 \begin{theorem}
\label{thm:Polytopal-graphs-cycle-2-face} For $d\ge 2$, every even subgraph in the graph of a $d$-polytope is the symmetric difference of facial cycles. In particular, every cycle in the graph of a $d$-polytope is the symmetric difference of facial cycles.
\end{theorem}    
%The proof of the theorem is a neat application of shellings of polytopes. 

A first corollary of \cref{thm:Polytopal-graphs-cycle-2-face} is immediate.
 
\begin{corollary} For $d\ge2$, the cycle space of the graph of a $d$-polytope is spanned by the facial cycles of the polytope. 
\label{cor:Polytopal-graphs-cycle-space}
\end{corollary}
	 
 As a second corollary of \cref{thm:Polytopal-graphs-cycle-2-face},  we obtain a characterisation of bipartite polytopal graphs, which was proved in \citet[Sec.~3]{Blind1994} via shellings of polytopes; this proof  is the inspiration for our proof of \cref{thm:Polytopal-graphs-cycle-2-face}. We remark that this characterisation was known to \citet[p.~154]{Cox73} but his proof was incorrect, as pointed out in \citet[Sec.~3]{Blind1994}.  
 
 \begin{corollary}[{\citealt{Blind1994}}]  
\label{cor:Polytopal-graphs-bipartite-2-face} For $d\ge 2$, the graph of a $d$-polytope is bipartite if and only if  every 2-face of the polytope is bipartite. 
\end{corollary}
\begin{proof} 
If the graph of a polytope is bipartite, then every facial cycle must have even length and so it is bipartite.  If a  graph $G$ of a polytope is nonbipartite, then $G$ has a cycle $C$ of odd length.   By \cref{cor:Polytopal-graphs-cycle-space}, the cycle $C$ (as an even subgraph) is the symmetric difference of facial cycles of $G$. Accordingly, one of these facial cycles must have odd length.
% A long argument goes as follows: assume that C is the shortest cycle with odd length in G. If C is not a facial cycle, then C can be written as the symmetric difference of two shorter cycles C1 and C2 (not necessarily facial) in its interior. So C1 and C2 have even length. But then their symmetric difference will have even length.
\end{proof}

\cref{thm:Polytopal-graphs-cycle-2-face}, \cref{cor:Polytopal-graphs-cycle-space}, and \cref{cor:Polytopal-graphs-bipartite-2-face} are known to the community of discrete geometry.   \cref{thm:Polytopal-graphs-cycle-2-face} is usually proved by homological arguments \citep{Nev22}, and \cref{cor:Polytopal-graphs-bipartite-2-face} follows from it, as we illustrated it.

%Unless otherwise stated, the graph theoretical notation and terminology follow from \cite{Die05} and the polytope theoretical notation and terminology from \cite{Zie95}.  

\section{Proofs}

A {\it  polytopal complex} $\mathcal C$ is a finite, nonempty collection of polytopes in $\R^{d}$ where the faces  of each polytope in $\mathcal C$ all belong to $\mathcal C$ and where polytopes intersect only at faces. The \textit{boundary complex} $\B(P)$ of a polytope $P$ is the set of faces of $P$ other than $P$ itself, while the \textit{complex} $\C(P)$ of $P$ is the set of faces of $P$. 
%A \textit{subcomplex} of a polytopal complex $\C$ is a subset of $\C$ that is itself a polytopal complex. The  {\it star} of a vertex $x$ in a polytope $P$ is the subcomplex of $\B(P)$ formed by all the faces containing $x$, and their faces. The  {\it link} of a vertex $x$ in $P$, denoted $\lk(x)$, is the subcomplex of $\B(P)$ formed by all the faces of the star of $x$ that are disjoint from $x$. 

 Let $\C$ be a \textit{pure polytopal complex}; that is, each of the faces of $\C$ is contained in some facet. A \textit{shelling} of $\C$ is a linear ordering $F_{1},\ldots,F_{s}$ of its facets such that either $\dim \C=0$, in which case the facets are vertices, or it satisfies the following:
\begin{enumerate}
\item The boundary complex of $F_{1}$ has a shelling.
\item For $2\le j\le s$, the intersection  
\[F_{j}\cap \left(\bigcup_{i=1}^{j-1} F_{i}\right)=R_{1}\cup \cdots\cup R_{r}\] 
is nonempty and the beginning $R_{1}, \ldots, R_{r}$ of a shelling $R_{1}, \ldots, R_{r},   R_{r+1},\ldots$, $ R_{t}$ of  $\B(F_{j})$.           
\end{enumerate}  
 
 \citet{BruMan71} proved that every polytope  admits a shelling. 
%A complex is \textit{shellable} if it is pure and admits a shelling. 

%The \textit{face poset} of a polytopal complex $\C$ is the poset formed by the faces of  $\C$, partially ordered by inclusion.   Two polytopal complexes are \textit{combinatorially isomorphic} if their face posets are isomorphic; two posets $\Lc$ and $\Lc'$ are \textit{isomorphic} if there is an order-preserving bijection $f:\Lc\to \Lc'$:  $x\le y$ is in $\Lc$ if and only if  $f(x)\le f(y)$ is in $\Lc'$.

 A pure polytopal complex $\C$ is {\it strongly connected} if every pair of facets $F$ and $F'$ is connected by a path $F_{1}\ldots F_{n}$ of facets in $\C$ such that $F_{i}\cap F_{i+1}$ is a ridge of $\C$ for $1\le i\le n-1$, $F_{1}=F$ and $F_{n}=F'$. This definition implies the following two assertions.%; we say that such a path is a {\it $(d-1,d-2)$-path} or a {\it facet-ridge path} if the dimensions of the faces can be deduced from the context.  

\begin{proposition} For $d\ge 1$, the graph of a strongly connected $d$-complex is connected.
\label{prop:Polytopal-graphs-strongly-connected-complex-connectivity} 
\end{proposition}

For a set $\{P_{1},\ldots,P_{n}\}$ of polytopes, we denote by $\C(P_{1}\cup \cdots \cup P_{n})$ the polytopal complex $\C(P_{1})\cup \cdots \cup \C(P_{n})$.

 \begin{proposition}
\label{prop:Polytopal-graphs-shelling-strongly-connected-complex} Let $P$ be a $d$-polytope, and let $F_{1},\ldots F_{s}$ be a shelling of $P$. Then the complex $\C(F_{1}\cup \cdots\cup F_{i})$ is a strongly connected $(d-1)$-complex, for each $1\le i\le s$. %In particular, the boundary complex of $P$ is a strongly connected $(d-1)$-complex. 
\end{proposition}

%\begin{proposition} The graph of a 3-polytope is planar and 3-connected.   
%\label{prop:Polytopal-graphs-3polytopes-planar}  
%\end{proposition}
 
	\begin{proof}[Proof of \cref{thm:Polytopal-graphs-cycle-2-face}] 
For $d= 2$, the theorem is trivially true. %For $d=3$ the theorem is a consequence of \cref{prop:Polytopal-graphs-3polytopes-planar,prop:Graphs-cycle-facial-cycles}; here we implicitly used the fact that each even subgraph is the symmetric difference of cycles (\cref{prop:Graphs-even-subgraphs-basic}). 
So assume that $d\ge 3$ and proceed by induction on $d$. 
    
Let $S:=F_{1},\ldots,F_{s}$ be a shelling of $P$, let \[\C_{n}:=\C( F_{1}\cup\cdots\cup F_{n}),\]  and let $G_{n}$ be the graph of $\C_{n}$. For $1\le i\le s-1$, each facet $F_{i}$ is a $(d-1)$-polytope, and so the induction hypothesis ensures that the  even subgraphs in each $G(F_{i})$ are generated by facial cycles in such a graph.   
      
Since all the edges of $P$ lie in $G_{s-1}$---and in particular, $G_{s-1}=G$---it suffices to prove that every even subgraph in each $G_{n}$ is spanned by the  facial cycles of $G_{n}$, for $1\le n\le s-1$.  We further proceed  by induction on $n$. The case $n=1$ holds, because $G_{1}=G(F_{1})$.   We then assume that each even subgraph of $G_{n-1}$, where $1\le n-1\le s-2$,  is spanned by the facial cycles in $G_{n-1}$ and that some \textit{problematic} even subgraph in $G_{n}$ is not. Because $G_{n}=G_{n-1}\cup G(F_{n})$, each problematic even subgraph in $G_{n}$ is contained in neither $G_{n-1}$ nor $G(F_{n})$.  Let $C$  be a problematic even subgraph in $G_{n}$ with a smallest number of edges in $G_{n-1}\setminus G(F_{n})$ among the problematic even subgraphs in $G_{n}$. Each even subgraph is the symmetric difference of cycles; and consequently, we have that $C$ is a cycle of $G_{n}$. 

%; if  $C$ didn't contain a vertex from $ G_{n-1}\setminus G(F_{n})$, then it would contain a vertex from $G(F_{n})\setminus G_{n-1}$. Without loss of generality, we may assume that   
 
  A  path from a vertex $x$ to a vertex $y$ in a graph is an {\it $x-y$ path}, and for a path $L:=x_{0}\ldots x_{n}$ and for $0\le i\le j\le n$, we write $x_{i}Lx_{j}$  to denote the subpath $x_{i}\ldots x_{j}$. 
%two $x-y$ paths  are {\it independent}\index{independent} if they share no inner vertex.

Think of $C$ as a cycle directed  from in $G_{n-1}\setminus G(F_{n})$ to $G(F_{n})$, starting at a vertex in $ G_{n-1}\setminus G(F_{n})$; the addition of $G(F_{n})$  to $G_{n-1}$ may introduce new edges and no new vertices.  The cycle $C$ intersects $G_{n-1}\cap G(F_{n})$ in the distinct vertices $\v x_{1},\ldots, \v x_{j}$ found in this order as we traverse $C$.  That is, $\v x_{1}$ is the first vertex on $C$ that touches $G_{n-1}\cap G(F_{n})$,  $\v x_{j}$ is the last vertex on $C$ that touches $G_{n-1}\cap G(F_{n})$, and the directed subpath $L:=\v x_{j} C \v x_{1}$ of $C$ lies in $G_{n-1}\setminus G(F_{n})$ except for $\v x_{1}$ and $\v x_{j}$.  If $G_{n-1}\cap G(F_{n})\cap C$ consists of one vertex, then $C$ would be a union of cycles with a cycle in $G_{n-1}$ and a cycle in $G(F_{n})$, a contradiction to $C$ being a cycle. Therefore $j\ge 2$. %In this way, the subgraph  $C-\left\{\v x, \v y\right\}$ is not fully contained in $G_{n-1}\cap G(F_{n})$. %Furthermore,  the lengths of the paths $\v x C \v y$ and  $\v y C \v x$ have distinct parity, as $C$ has odd length.
 
 The graph $G_{n-1}\cap G(F_{n})$  is connected. Since $S$ is a shelling, the complex  $\B(F_{n})\cap \C_{n-1}$ is the beginning of a shelling of $F_{n}$, which implies that $\B(F_{n})\cap \C_{n-1}$ is a strongly connected $(d-2)$-complex (\cref{prop:Polytopal-graphs-shelling-strongly-connected-complex}); the connectivity of $G_{n-1}\cap G(F_{n})$ now follows from  \cref{prop:Polytopal-graphs-strongly-connected-complex-connectivity}. 
 
 From the connectivity of $G_{n-1}\cap G(F_{n})$ follows the existence of an $\v x_{1}-\v x_{j}$ path $M$ in $G_{n-1}\cap G(F_{n})$.  Concatenating the paths $L$ and $M$, we form a cycle $C_{1}$ in $G_{n-1}$. By the induction hypothesis on $n$, this cycle $C_{1}$ is the symmetric difference of facial cycles in $G_{n-1}$. Let $L'$ be the directed subpath $L'$ of $C$ from $\v x_{1}$ to $\v x_{j}$. Then $C=L\cup L'$ and $V(L)\cap V(L')=\lset{\v x_{1},\v x_{j}}$. In addition, let $W:=L'\sdiff M$; here we understand $L'$ and $M$ as spanning subgraphs of $G_{n}$. It follows that $W$ is an even subgraph of $G$, since $\v x_{1}$ and $ \v x_{j}$  have each degree one in both $L'$ and $W$, and every other vertex in $W$ has even degree in both $L'$ and $W$. It is also the case that  $W$ has fewer edges in $G_{n-1}\setminus G(F_{n})$ than $C$, and so it is the symmetric difference of facial cycles in $G_{n}$. 
 
The cycle $C$ is the symmetric difference of $C_{1}$ and $W$, and as a consequence, it is the symmetric difference of facial cycles in $G_{n}$. This contradiction ensures that the cycle $C$ does not exist, which amounts to saying that every even subgraph in $G_{n}$ is spanned by facial cycles in $G_{n}$. Hence the induction is complete, and so is the proof of the theorem.   
 \end{proof}

\section{Acknowledgements} 

The author wants to thank Julien Ugon and Eran Nevo for helpful comments on the manuscript.

\end{document}